\documentclass[11pt, a4paper]{amsart}

\usepackage{amsmath,amssymb,hyperref}
\usepackage{amsfonts}
\usepackage{enumerate}
\usepackage{amsmath, amsthm}
\usepackage{amscd}
\input xy
\xyoption{all}

      \theoremstyle{plain}
      \newtheorem{theorem}{Theorem}[section]
      \newtheorem{lemma}[theorem]{Lemma}
      \newtheorem{corollary}[theorem]{Corollary}
      \newtheorem{proposition}[theorem]{Proposition}
      \theoremstyle{definition}
      \newtheorem{definition}[theorem]{Definition}

      \theoremstyle{remark}

\newcommand{\su}{{\mathrm{SU}(n,1)}}

\newcommand{\Sp}{\mathrm{Sp}(n,1)}

\newcommand{\hH}{\mathbf{H}^n_{\mathbb H}}

\def\C{\mathbb C}
\def\R{\mathbb R}

\def\H{\mathbb H}
\def\P{\mathbb P}

\def\<{\langle}
\def\>{\rangle}

\def\0{\mathbf{0}}
\def\hh#1{{{\bf H}^{#1}_{\H}}}
\def\[#1\]{\begin{eqnarray*}#1\end{eqnarray*}}

      \makeatletter
      \def\@setcopyright{}
      \def\serieslogo@{}
      \makeatother

\begin{document}

%



   \author{Sungwoon Kim}
   \address{Department of Mathematics, Jeju National University, Jeju, 63243, Republic of Korea}
   \email{sungwoon@jejunu.ac.kr}

   \author{Joonhyung Kim}
   \address{Department of Mathematics Education, Jeonju University, Jeonju, 55069, South Korea}
   \email{calvary@snu.ac.kr}






   \title[Complex trace fields]{Quaternionic hyperbolic Kleinian groups with commutative trace skew-fields}


\begin{abstract}
Let $\Gamma$ be a nonelementary discrete subgroup of $\Sp$. We show that if the trace skew-field of $\Gamma$ is commutative, then $\Gamma$ stabilizes a copy of complex hyperbolic subspace of $\hH$.
\end{abstract}



\footnotetext[3]{The second author was supported by the Basic Science Research Program through the National Research Foundation of Korea (NRF) funded by the Ministry of Education, Science and Technology (NRF-2017R1C1B1003906)}




   \thanks{}
   \thanks{}

   \dedicatory{}

   \date{}


   \maketitle



\section{Introduction}

The trace field of a linear group is defined as the (skew) field generated by the traces of its elements.
The property of algebraic or geometric nature of a linear group is frequently reflected in its trace fields.
For instance, Neumann and Reid \cite{NR} proved that a nonuniform arithmetic lattice of $\mathrm{PSL}(2,\C)$
is realized over its trace field. Cunha and Gusevskii \cite{CG}, and Genzmer \cite{Ge} extended Neumann and Reid's result to some subgroups of $\mathrm{SU}(2,1)$.
These results are concerned with algebraic aspects reflected in trace fields. On the other hand, there have been many studies on geometric aspects reflected in trace fields.
Maskit \cite{Mas98} showed that if the trace field of a subgroup of $\mathrm{SL}(2,\C)$ is real, the subgroup preserves a totally geodesic subspace isometric to $\mathbf H^2_{\R}$ in $\mathbf H_{\R}^3$.
The same question concerning real trace field naturally arises in the simple Lie groups of $\su$ and $\Sp$. 
At first, in the case of $\mathrm{SU}(2,1)$, it turns out that a nonelementary discrete subgroup with real trace field stabilizes a real hyperbolic subspace $\mathbf H^2_{\R}$ of $\mathbf H^2_{\C}$ in \cite{CG,FLW12}. This result is extended to $\mathrm{SU}(3,1)$ in \cite{KK14} and moreover $\mathrm{Sp}(2,1)$ in \cite{Kim13}.
In the end, J. Kim and S. Kim \cite{KK16} answered the question for general simple Lie groups of rank $1$.
Precisely speaking, they \cite{KK16} prove that if the trace field of a nonelementary discrete subgroup of $\su$ or $\Sp$ is real, the group stabilizes a totally geodesic submanifold of constant negative sectional curvature. Note that such totally geodesic submanifold of constant negative sectional curvature is isometric to $\mathbf H^k_{\R}$ for some $2\leq k\leq n$ or $\mathbf H^1_{\C}$.

While geometric aspects reflected in real trace fields have been intensively studied, there have been no studies on commutative trace skew-fields of subgroups of $\Sp$. 
Recently, J. Kim and S. Kim \cite{KK18} showed that if a nonelementary discrete subgroup $\Gamma$ of $\mathrm{Sp}(2,1)$ has a commutative trace skew-field, it is conjugate to a subgroup of $\mathrm{U}(2,1)$.
In other words, it stabilizes a totally geodesic submanifold isometric to $\mathbf H^2_{\C}$. 

In general, the trace skew-field of a subgroup of $\Sp$ might be not commutative. Note that the field of complex numbers is one of maximal commutative skew-subfields of $\H$.
In the paper, we figure out what geometric property is reflected in commutative trace skew-fields as follows.

\begin{theorem}\label{thm:main}
Let $\Gamma$ be a nonelementary discrete subgroup of $\Sp$. If the trace skew-field of $\Gamma$ is commutative, then $\Gamma$ stabilizes a totally geodesic submanifold isometric to $\mathbf H^k_{\C}$ in $\hH$  for $1\leq k\leq n$.
\end{theorem}

As a corollary, we have the following.

\begin{theorem}\label{cor:1}
Let $\Gamma$ be an irreducible subgroup of $\Sp$ such that the trace skew-field of $\Gamma$ is commutative. Then $\Gamma$ is conjugate to a subgroup of $\mathrm{U}(n,1)$.
\end{theorem}

\section{Preliminaries}

In this section, we briefly review necessary background.

\subsection{Quaternionic hyperbolic spaces}
Let $\H^{n,1}$ be a quaternionic vector space of dimension $n+1$ with a
Hermitian form of signature $(n,1)$. An element of $\H^{n,1}$ is a
column vector $p=(p_1,\ldots,p_{n+1})^t$. As in the complex hyperbolic case, we choose
the Hermitian form on $\H^{n,1}$ given by the matrix $I_{n,1}$
$$
I_{n,1}=\left[\begin{matrix} I_n & 0 \\ 0 & -1
\end{matrix}\right].
$$
Thus $\<p,q\>=q^*I_{n,1}p=\overline{q}^tI_{n,1}p=\overline{q}_1p_1+\overline{q}_2p_2+\cdot\cdot\cdot+\overline{q}_np_n-\overline{q}_{n+1}p_{n+1}$,
where $p=(p_1,\ldots,p_{n+1})^t$, $q=(q_1,\ldots,q_{n+1})^t \in \H^{n,1}$. The group $\Sp$ is the subgroup of $\mathrm{GL}(n+1,\mathbb H)$ which, when acting on the left, preserves the Hermitian form given above.

Let $\P:\H^{n,1}\setminus\{\0\} \rightarrow {\H}P^{n}$ be the
canonical projection onto a quaternionic projective space. Consider the
following subspaces in $\H^{n,1}$;
\begin{align*}
V_0 &= \{ z\in {\H}^{n,1}- \{ \0 \}\ \ |\ \ \<z,z\>=0 \ \},\\
V_- &= \{ z\in {\H}^{n,1}\ \ |\ \ \<z,z\> < 0 \ \}.
\end{align*}
The $n$-dimensional \emph{quaternionic hyperbolic space $\hH$} is
defined as $\P(V_-)$. The \emph{boundary} $\partial\hH$ is
defined as $\P(V_0)$.
There is a metric on $\hH$ called the Bergman metric and the isometry group of $\hh{n}$ with respect to this metric is
\begin{align*}
\mathrm{PSp}(n,1) &=\{[A]:A \in \mathrm{GL}(n+1,\H), \langle p,p' \rangle = \langle Ap,Ap' \rangle, p,p' \in \H^{n,1}\}\\
&= \{[A]:A \in \mathrm{GL}(n+1,\H), I_{n,1}=A^*I_{n,1}A\},
\end{align*}
where $[A]: {\H}P^{n} \rightarrow {\H}P^{n}; x\H \mapsto (Ax)\H$ for $A \in \mathrm{Sp}(n,1)$. Here we adopt the convention that the action of $\mathrm{Sp}(n,1)$ on $\hH$ is left and the action of projectivization of $\mathrm{Sp}(n,1)$ is right action. In fact $\mathrm{PSp}(n,1)$ is the quotient group by the real scalar matrices in $\Sp$. Thus it is not difficult to see that $$\mathrm{PSp}(n,1)=\Sp / \{\pm  I \}.$$

Similarly to the complex hyperbolic space, totally geodesic submanifolds of quaternionic hyperbolic space are isometric to  either $\hh{k}$, $\mathbf{H}^k_{\mathbb C}$ or $\mathbf{H}^k_{\mathbb R}$ for some $1\leq k\leq n$. Note that a totally geodesic submanifold of constant negative sectional curvature is isometric to either $\mathbf H_{\mathbb R}^k$ for some $2\leq k \leq n$, $\mathbf H_{\mathbb C}^1$ or $\mathbf H_{\mathbb H}^1$. 
The classification of isometries by their fixed points is exactly the same as in the complex hyperbolic case.

\begin{definition}
Let $\Gamma$ be a subgroup of $\Sp$. Then the \emph{trace skew-field} of $\Gamma$, denoted by $\mathbb Q(\mathrm{tr}\Gamma)$, is defined as the skew field generated by the traces of all the elements of $\Gamma$ over the base field $\mathbb{Q}$ of rational numbers.
\end{definition}

We say that the trace skew-field of $\Gamma$ is \emph{commutative} if all the elements of the trace skew-field of $\Gamma$ commute.

\subsection{Zariski topology}\label{sec:zariski}

Let $\R[x_{1,1},\ldots,x_{n,n}]$ denote the set of real polynomials in the $n^2$ variables $\{x_{j,k} \ | \ 1\leq j,k \leq n \}$.
A subset $H$ of $\mathrm{SL}(n,\R)$ is called \emph{Zariski closed} if there is a subset $\mathcal S$ of $\R[x_{1,1},\ldots,x_{n,n}]$ such that $H$ is the zero locus of $\mathcal S$. In particular, when $H$ is a subgroup of $\mathrm{SL}(n,\R)$, $H$ is called a \emph{real algebraic group}.
It is a standard fact that any Zariski closed subset of $\mathrm{SL}(n,\R)$ has only finitely many components. Furthermore, a Zariski closed subgroup of $\mathrm{SL}(n,\R)$ is a $C^\infty$-submanifold of $\mathrm{SL}(n,\R)$, hence a Lie group.

\begin{definition} The \emph{Zariski closure} of a subset $H$ of $\mathrm{SL}(n,\R)$ is the (unique) smallest Zariski closed subset of $\mathrm{SL}(n,\R)$ that contains $H$. We use $\overline H$ to denote the Zariski closure of $H$.
\end{definition}

It is well-known that if $H$ is a subgroup of $\mathrm{SL}(n,\R)$, then $\overline H$ is also a subgroup of $\mathrm{SL}(n,\R)$.

\begin{definition}
A subgroup $H$ of $\mathrm{SL}(n,\R)$ is \emph{almost Zariski closed} if $H$ is a finite-index subgroup of $\overline H$.
\end{definition}

We remark that a connected subgroup $H$ of $\mathrm{SL}(n,\R)$ is almost Zariski closed if and only if it is the identity component of a Zariski closed subgroup.

\subsection{Simple Lie subgroups of $\Sp$}\label{sec:simple}

Let $H$ be a noncompact semisimple Lie subgroup of $\Sp$ with Lie algebra $\mathfrak h \subset \mathfrak{sp}(n,1)$.
Then since the real rank of $\mathfrak{sp}(n,1)$ is $1$, all possible types for $\mathfrak h$ are listed as follows:
 $$\mathfrak{so}(m,1), \ \mathfrak{su}(k,1), \ \mathfrak{sp}(k,1) \text{ for }m=2,\ldots,n \text{ and }k=1,\ldots,n.$$
Indeed, $o_{n-k}\oplus\mathfrak{so}(k,1), \ o_{n-k}\oplus\mathfrak{su}(k,1)$ and $\ o_{n-k}\oplus\mathfrak{sp}(k,1)$ are subalgebras of $\mathfrak{sp}(n,1)$ for $1\leq k\leq n$
where $o_{n-k}$ denotes the zero square matrix of size $n-k$.
For easy of notation, hereafter we write $\mathfrak{so}(k,1), \ \mathfrak{su}(k,1)$ and $\mathfrak{sp}(k,1)$ for $o_{n-k}\oplus\mathfrak{so}(k,1), \ o_{n-k}\oplus\mathfrak{su}(k,1)$ and $\ o_{n-k}\oplus\mathfrak{sp}(k,1)$ respectively.

It is well known that there exists a unique connected Lie subgroup $H$ of $G$ whose Lie subalgebra of $G$ is $\mathfrak h$.
Hence $I_{n-k}\oplus \mathrm{SO}(k,1)^\circ, I_{n-k}\oplus \mathrm{SU}(k,1)$ and $I_{n-k}\oplus \mathrm{Sp}(k,1)$ are the unique connected Lie subgroups of $\Sp$ whose Lie subalgebras of $G$ are $\mathfrak{so}(k,1), \ \mathfrak{su}(k,1)$ and $\mathfrak{sp}(k,1)$ respectively where $\mathrm{SO}(k,1)^\circ$ is the identity component of $\mathrm{SO}(k,1)$. Note that $\mathrm{SU}(k,1)$ and $\mathrm{Sp}(k,1)$ are connected but $\mathrm{SO}(k,1)$ is not connected for all $k\geq 1$.


\section{Proof}\label{sec:proof}

We start with the observation that any maximal commutative skew-subfield of the quaternions $\H$ is similar to $\C$.

\begin{lemma}\label{lem:mas}
Let $F$ be a maximal commutative skew-subfield of $\H$. Then there exists a unit quaternion $q\in \H$ such that $qF\bar q=\C$.
\end{lemma}
\begin{proof}
First observe that $\R$ must be contained in any maximal commutative skew-subfield of $\H$ since $\R$ is the center of $\H$. With this observation, one can easily see that $F$ is a vector space over $\R$.
Choose a non-real number $u\in F$. Since any quaternion is similar to a complex number, there exists a unit quaternion $q \in \H$ with $qu\bar q \in \C$.
Clearly, $qF\bar q$ is again a maximal commutative skew-subfield of $\H$ and moreover, it contains a complex number $qu\bar q$ that is not real. By the observation in the beginning of the proof, one can see that the imaginary unit $i$ is contained in $qF\bar q$.
Furthermore, it is straightforward to show that if a quaternion commutes with $i$, it should be a complex number. Therefore we conclude that $qF\bar q=\C$.
\end{proof}

Let $\Gamma$ be a nonelementary discrete subgroup of $\Sp$ whose trace skew-field is commutative.
Then the trace skew-field is contained in a maximal commutative skew-subfield $F$ of $\H$.
By Lemma \ref{lem:mas}, there is a unit quaternion $q \in \H$ such that $q F \bar q=\C$. Let $Q$ be the diagonal matrix of size $n+1$ whose diagonal entries are all $q$.
Then $Q\in \Sp$ and the trace skew-field of $Q\Gamma Q^{-1}$ is a subfield of $\C$.
In other words, by conjugation, we may assume that the trace skew-field of $\Gamma$ is contained in $\C$.

\subsection{Embedding of $\Sp$ into $\mathrm{SL}(4n+4,\mathbb R)$}
The correspondence $$a+bi+cj+dk \mapsto \left[ \begin{array}{rrrr} a & b & c & d \\ -b & a & -d & c \\ -c & d & a & -b \\ -d & -c & b & a \end{array} \right]$$ 
induces a homomorphism $\theta : \Sp \rightarrow \mathrm{GL}(4n+4,\mathbb R)$. It is easy to check that $\theta$ is an injective homomorphism and $\theta(g^*)=\theta(g)^t$.
Hence the relation $g^*I_{n,1}g=I_{n,1}$ implies that 
\begin{align*}
\det(\theta(g^*)\theta(I_{n,1})\theta(g))&=\det(\theta(g)^t)\det(\theta(g)) \\ &=\det(\theta(g))^2 \\&=\det(\theta(I_{n,1}))=1.
\end{align*}
This means that for any $g\in \Sp$ the determinant of $\theta(g)$ is either $1$ or $-1$.
Since $\Sp$ is connected and the determinant function is continuous, it follows that $\det(\theta(g))=1$ for all $g\in \Sp$.
Thus $\theta$ is an embedding of $\Sp$ into $\mathrm{SL}(4n+4,\mathbb R)$.

\subsection{Matrices with complex traces}

For an element $g$ of $\Sp$, define the trace of $g$, denoted by $tr(g)$, as the sum of diagonal entries of $g$.
We remark that the trace is not invariant under conjugation in $\Sp$. 
Define a subset $\mathrm{Tr}(\C)$ of $\Sp$ by $$\mathrm{Tr}(\C)=\{ g\in \Sp \ | \ tr(g) \in \mathbb C\}.$$
Let $d_m=a_m+b_mi+c_mj+d_mk$ be the $(m,m)$-entry of $g \in \Sp$ for $1\leq m\leq n+1$.
Then $$tr(g)\in \mathbb C \Longleftrightarrow\sum_{m=1}^{n+1} c_m = \sum_{m=1}^{n+1} d_m=0.$$
From this observation, it follows that $\theta(\mathrm{Tr}(\C))$ is a Zariski closed subset of $\mathrm{SL}(4n+4,\mathbb R)$.

Since the trace of each element of $\Gamma$ is a complex number, $\theta(\Gamma) \subset \theta(\mathrm{Tr}(\C))$.
To ease notation, we write $\Gamma_\theta=\theta(\Gamma)$ and $\mathrm{Tr}_\theta(\C)=\theta(\mathrm{Tr}(\C))$.
The set $\mathrm{Tr}_\theta(\C)$ is Zariski closed and hence the Zariski closure $\overline{\Gamma}_\theta$ of $\Gamma_\theta$ is a subset of $\mathrm{Tr}_\theta(\C)$.
This means that the trace of every element of $\overline{\Gamma}_\theta$ is also a complex number.

\subsection{Structure of almost Zariski closed groups} \label{sec:szcg}
The Zariski closure $\overline{\Gamma}_\theta$ is a Zariski-closed subgroup of $\mathrm{SL}(4n+4,\mathbb R)$ with finitely many connected components.
Thus the identity component $\overline{\Gamma}_\theta^\circ$ of $\overline{\Gamma}_\theta$ is an almost Zariski closed subgroup of $\mathrm{SL}(4n+4,\mathbb R)$.
Applying Theorem 4.4.7 in \cite{Mor} for the structure of almost Zariski closed groups, there exist
\begin{itemize}
\item a semisimple subgroup $L$ of $\overline{\Gamma}_\theta^\circ$,
\item a torus $T$ in $\overline{\Gamma}_\theta^\circ$, and
\item a unipotent subgroup $U$ of $\overline{\Gamma}_\theta^\circ$,
\end{itemize}
such that
\begin{itemize}
\item $\overline{\Gamma}_\theta^\circ=(LT)\ltimes U$,
\item $L, T$, and $U$ are almost Zariski closed, and
\item $L$ and $T$ centralize each other and have finite intersection.
\end{itemize}

Let $H$ be a noncompact simple factor of $L$. If there are no noncompact simple factors of $L$, then $L$ is compact and hence $\overline{\Gamma}_\theta^\circ$ is amenable.
This implies that $\Gamma$ is also amenable, which contradicts the assumption that $\Gamma$ is nonelementary.
Thus there is a noncompact simple factor $H$ of $L$.
The Lie algebra $\mathfrak h$ of $H$ is isomorphic to one of the following.
 $$\mathfrak{so}(m,1), \ \mathfrak{su}(k,1), \ \mathfrak{sp}(k,1) \text{ for }m=2,\ldots,n \text{ and }k=1,\ldots,n.$$
Observing noncompact simple Lie subgroups of $\Sp$ in Section \ref{sec:simple}, it follows that $H$ is isomorphic to one of the following.
 $$\mathrm{SO}(k,1)^\circ, \ \mathrm{SU}(k,1), \ \mathrm{Sp}(k,1) \text{ for }m=2,\ldots,n \text{ and }k=1,\ldots,n.$$
The condition that the trace of every element of $H$ is a complex number will exclude the case where $H$ is isomorphic to $\mathrm{Sp}(k,1)$ for $1\leq k\leq n$. To prove this, we start with the following Proposition.

\begin{proposition}\label{prop:sp}
Let $1\leq k\leq n$.
There is no element $g \in \Sp$ such that every element of $g \left( I_{n-k}\oplus \mathrm{Sp}(k,1) \right)g^{-1}$ has its trace a complex number.
\end{proposition}

\begin{proof}
To obtain a contradiction, we suppose that for some $g\in \Sp$, the trace of every element of $g \left( I_{n-k}\oplus \mathrm{Sp}(k,1) \right) g^{-1}$ is a complex number.
Let $a_{p,q}$ denote the $(p,q)$-entry of $g$ for $1\leq p, q \leq n+1$. Since $g$ satisfies the equation $g^* I_{n,1}g=I_{n,1}$, the inverse $g^{-1}$ of $g$ is written as
$$g^{-1} = \left[ \begin{array}{cccc} a_{1,1}^* & \cdots & a_{n,1}^* & - a_{n+1, 1}^* \\ \vdots & \ddots & \vdots &\vdots \\ a_{1,n}^* & \cdots & a_{n,n}^* & -a_{n+1, n}^* \\ -a_{1, n+1}^* & \cdots & - a_{n, n+1}^* & a_{n+1,n+1}^* \end{array} \right].$$

Let $j_n$ be the diagonal matrix  of size $n+1$ with diagonal entries $1,\ldots,1, j$ and $k_n$ be the diagonal matrix  of size $n+1$ with diagonal entries $1,\ldots,1,k$.
Obviously $j_n$ and $k_n$ are elements of $I_{n-k}\oplus \mathrm{Sp}(k,1)$ for any $1\leq k\leq n$.
By a straight computation, the trace of $g j_n g^{-1}$ is 
\begin{align}
\sum_{m=1}^n \sum_{l=1}^n & \|a_{m,l}\|^2 -\sum_{m=1}^n\|a_{n+1,m}\|^2 \label{eqn:trgjg1} \\ 
&-\sum_{m=1}^n a_{m,n+1}   j  a_{m,n+1}^* + a_{n+1,n+1}j  a_{n+1,n+1}^*. \label{eqn:trgjg2}
\end{align}
By assumption, the trace of $g j_n g^{-1}$ is a complex number. Every term in (\ref{eqn:trgjg1}) is real and thus $tr(gj_ng^{-1}) \in \mathbb C$ is equivalent to 
\begin{equation}\label{eqn:1}
\sum_{m=1}^n a_{m,n+1}  j  a_{m,n+1}^* - a_{n+1,n+1} j  a_{n+1,n+1}^* \in \mathbb C.
\end{equation}
Similarly, it follows from a straightforward computation that $tr(g k_n g^{-1})\in \mathbb C$ is equivalent to 
\begin{equation}\label{eqn:2}
\sum_{m=1}^n a_{m,n+1}   k a_{m,n+1}^* - a_{n+1,n+1} k  a_{n+1,n+1}^* \in \mathbb C.
\end{equation}
Furtheremore, the identity $g^* I_{n,1}g=I_{n,1}$ gives us that  
\begin{equation}\label{eqn:3}
\sum_{m=1}^n \| a_{m,n+1}\|^2 - \| a_{n+1,n+1}\|^2 =-1.
\end{equation}

Set $a_{m,n+1}=x_{m,1}+x_{m,2}i+x_{m,3}j+x_{m,4}k$ for $1\leq m\leq n+1$. Then it is easy to see that
\begin{align*}
a_{m,n+1} j  a_{m,n+1}^* = 2(&x_{m,2}  x_{m,3}-x_{m,1}x_{m,4})i +(x_{m,1}^2-x_{m,2}^2+x_{m,3}^2-x_{m,4}^2)j \\ &+2(x_{m,1}x_{m,2}+x_{m,3}x_{m,4})k
\end{align*}
and, 
\begin{align*}
a_{m,n+1} k  a_{m,n+1}^* = 2(&x_{m,1}  x_{m,3}+ x_{m,2}x_{m,4})i  
+2(-x_{m,1}x_{m,2}+x_{m,3}x_{m,4})j\\ &+(x_{m,1}^2-x_{m,2}^2-x_{m,3}^2+x_{m,4}^2)k. 
\end{align*}
The condition (\ref{eqn:1}) means that the $j$-part and $k$-part of the term in (\ref{eqn:1}) are all zero. Together with the identities for $a_{m,n+1} j  a_{m,n+1}^* $, and $a_{m,n+1} k  a_{m,n+1}^*$ above, we get the following equations:
\begin{align*}
& \sum_{m=1}^n \left( x_{m,1}^2-x_{m,2}^2+x_{m,3}^2-x_{m,4}^2\right)- \left(x_{n+1,1}^2-x_{n+1,2}^2+x_{n+1,3}^2-x_{n+1,4}^2\right)=0, \\
&  \sum_{m=1}^n 2\left( x_{m,1}x_{m,2}+x_{m,3}x_{m,4}\right)- 2\left( x_{n+1,1}x_{n+1,2}+x_{n+1,3}x_{n+1,4}\right)=0,\\
&  \sum_{m=1}^n 2\left( -x_{m,1}x_{m,2}+x_{m,3}x_{m,4}\right)- 2\left(- x_{n+1,1}x_{n+1,2}+x_{n+1,3}x_{n+1,4}\right)=0, \\
& \sum_{m=1}^n \left( x_{m,1}^2-x_{m,2}^2-x_{m,3}^2+x_{m,4}^2\right)- \left(x_{n+1,1}^2-x_{n+1,2}^2-x_{n+1,3}^2+x_{n+1,4}^2\right)=0.
\end{align*}

Let $\mathbb R^{n,1}$ be the usual Lorentzian space with the Lorentzian inner product $\langle \cdot, \cdot\rangle_{n,1}$ defined by 
$$\langle x,y\rangle_{n,1}= x_1y_1+\cdots+x_ny_n-x_{n+1}y_{n+1}$$
for vectors $x=(x_1,\ldots,x_{n+1})$,  $y=(y_1,\ldots,y_{n+1})$.
The squared norm of a vector $x=(x_1,\ldots,x_{n+1})$ in the Lorentzian space is written as 
$$\|x\|_{n,1}^2=x_1^2+\cdots+x_n^2-x_{n+1}^2.$$

Set $v_m=(x_{1,m},\ldots,x_{n+1,m})$ for $m=1,2,3,4$. Then the above four equations are reformulated as follows:
\begin{equation}\label{eqnsys_spn}
\begin{aligned}
& \|v_1 \|_{n,1}^2 - \| v_2\|_{n,1}^2 + \| v_3 \|_{n,1}^2-\| v_4\|_{n,1}^2=0, \\
& \langle v_1, v_2 \rangle_{n,1}+ \langle v_3, v_4 \rangle_{n,1}=0,\\
& \langle v_1, v_2 \rangle_{n,1}- \langle v_3, v_4 \rangle_{n,1}=0, \\
& \|v_1 \|_{n,1}^2 - \| v_2\|_{n,1}^2 - \| v_3 \|_{n,1}^2+\| v_4\|_{n,1}^2=0. 
     \end{aligned} 
\end{equation}
In addition, from (\ref{eqn:3}),
\begin{equation*}
\|v_1 \|_{n,1}^2 + \| v_2\|_{n,1}^2 + \| v_3 \|_{n,1}^2+\| v_4\|_{n,1}^2=-1.
\end{equation*}
Solving all equations simultaneously provides the following results.
\begin{align}
& \| v_1\|_{n,1}^2 =\| v_2\|_{n,1}^2, \ \| v_3 \|_{n,1}^2 =\| v_4\|_{n,1}^2, \ 2\left(\|v_1\|^2+\|v_3\|^2\right)=-1, \label{eqn:norm}\\ 
& \langle v_1, v_2 \rangle_{n,1} = \langle v_3, v_4 \rangle_{n,1}=0. \label{eqn:perp}
\end{align}
Due to $ 2\left(\|v_1\|^2+\|v_3\|^2\right)=-1$, either $v_1$ or $v_3$ has a negative Lorentzian norm. If $v_1$ has a negative Lorentzian norm, so does $v_2$ by (\ref{eqn:norm}). Moreover $v_2 \in v_1^\bot$ by (\ref{eqn:perp}). However this contradicts the fact that every vector perpendicular to a negative vector in the Lorentzian space has a positive Lorentzian norm.
In the case that $v_3$ has a negative Lorentzian norm, we also get a similar contradiction. Therefore for any $g\in \Sp$, the set of traces of elements of $g \left( I_{n-k}\oplus \mathrm{Sp}(k,1) \right) g^{-1}$ can not be contained in $\mathbb C$.
\end{proof}

As a corollary, we exclude the case that $H$ is isomorphic to $\mathrm{Sp}(k,1)$ as follows.

\begin{corollary}\label{cor:sp}
The noncompact simple factor of $\overline{\Gamma}_\theta^\circ$ is not isomorphic to $\mathrm{Sp}(k,1)$ for any $1\leq k\leq n$.
\end{corollary}

\begin{proof}
Suppose, to derive a contradiction, that $H$ is isomorphic to $\mathrm{Sp}(k,1)$ for some $1\leq k\leq n$.
Since all Lie subgroups of $\Sp$ isomorphic to $\mathrm{Sp}(k,1)$ are conjugate to each other, there is an element $g \in \Sp$ such that 
$$H= \theta\left( g \left( I_{n-k}\oplus \mathrm{Sp}(k,1) \right) g^{-1} \right) \text{ and } H\subset \mathrm{Tr}_\theta(\C).$$
However Proposition \ref{prop:sp} leads to the contradiction that any Lie subgroup of $\Sp$ that is isomorphic to $\mathrm{Sp}(k,1)$ can not be contained in $\mathrm{Tr}(\C)$, which finishes the proof.
\end{proof}

We now turn to the unipotent subgroup $U$ in the decomposition $\overline{\Gamma}_\theta^\circ=(LT)\ltimes U$.

\begin{lemma}\label{lem:unip}
The unipotent subgroup $U$  in the decomposition $\overline{\Gamma}_\theta^\circ=(LT)\ltimes U$ is trivial.
\end{lemma}
\begin{proof}
We first prove that every element of $U$ is the $\theta$-image of a parabolic isometry of $\hH$.
By the Borel-Tits theorem, there is a parabolic subgroup $P$ of $\theta(\Sp)$ such that the unipotent subgroup $U$ of $\theta(\Sp)$ is contained in the unipotent radical $N$ of $P$.
Then $P$ admits the Langlands decomposition $P=MAN$, where $A$ is the $\mathbb R$-split torus and $N$ is the unipotent radical of $P$. 
In particular, for some $a_\theta\in A$, we have
$$N=\left\{ g_\theta\in \theta(\Sp) \ \Big| \ \lim_{m\rightarrow \infty} a_\theta^{-m}g_\theta a^m_\theta=e_\theta\right\},$$
where $e_\theta$ is the identity element of $\theta(\Sp)$.
Putting $\theta(a)=a_\theta$, $\theta(g)=g_\theta$ and $\theta(e)=e_\theta$,  
$$\lim_{m\rightarrow \infty} a_\theta^{-m}g_\theta a^m_\theta=\lim_{m\rightarrow \infty} \theta(a^{-m}g a^m)=\theta(e).$$
Since $\theta$ is an embedding, $\lim_{m\rightarrow \infty} a^{-m}g a^m=e$. This implies that the $\theta$-preimage of every element of $N$ is a parabolic isometry of $\hH$, thereby showing that the $\theta$-preimage of every element of $U$ is parabolic, as desired.

To obtain a contradiction, suppose that $U$ is not trivial. 
Let $u_\theta$ be a nontrivial element of $U$. Since the $\theta$-preimage $u$ of $u_\theta$ is a parabolic isometry of $\hH$, there is only one fixed point $\xi$ of $u$ on $\partial \hH$.
Furthermore, the $\theta$-preimage of every element of $U$ fixes the point $\xi$ uniquely.
Noting that $U$ is a normal subgroup of $\overline{\Gamma}_\theta^\circ$, it easily follows that every element in the $\theta$-preimage of $\overline{\Gamma}_\theta^\circ$ must fix the point $\xi$.
This means that  the $\theta$-preimage of $\overline{\Gamma}_\theta^\circ$ is contained in the stabilizer subgroup of $\xi$ in $\Sp$ and thus $\xi$ is an $\Gamma$-invariant point.
It contradicts the assumption that $\Gamma$ is nonelementary. Therefore $U$ must be trivial.
\end{proof}

\begin{proof}[Proof of Theorem \ref{thm:main}]
From corollary \ref{cor:sp}, it follows that the noncompact simple factor $H$ of $L$ in the decomposition $\overline{\Gamma}_\theta^\circ=(LT)\ltimes U$ is isomorphic to either $\mathrm{SO}(k,1)^\circ$ or $\mathrm{SU}(k,1)$ for some $1\leq k\leq n$.
Thus the $\theta$-preimage of $H$ preserves either a real hyperbolic $k$-subspace or a complex hyperbolic $k$-subspace of $\hH$.
It is well known that every real hyperbolic $k$-subspace is contained in a complex hyperbolic $k$-subspace. We may thus assume that the $\theta$-preimage of $H$ preserves a complex hyperbolic $k$-subspace $\mathbf H_{\mathbb C}^k$ of $\hH$. 
Then the $\theta$-preimage of every simple factor of $L$ preserves $\mathbf H_{\mathbb C}^k$ since $H$ and any other simple factor of $L$ centralize each other.
Similarly the $\theta$-preimage of torus $T$ in $\overline \Gamma^\circ_\theta$ also preserves $\mathbf H_{\mathbb C}^k$.
In the end, the $\theta$-preimage of $\overline \Gamma^\circ_\theta$ preserves $\mathbf H_{\mathbb C}^k$.
Since $\overline \Gamma^\circ_\theta$ is a finite index subgroup of $\overline \Gamma_\theta$, the $\theta$-preimage of $\overline \Gamma_\theta$ stabilizes $\mathbf H_{\mathbb C}^k$ either and so does $\Gamma$.
\end{proof}

\begin{proof}[Proof of Theorem \ref{cor:1}]
By Theorem \ref{thm:main}, $\Gamma$ preserves a complex hyperbolic $k$-subspace in $\hH$. By conjugation, we may assume that $\Gamma$ preserves the complex hyperbolic $k$-subspace $\mathbf H_{\mathbb C}^k$ defined as
$$ \mathbb P\left( \{ (0,\ldots,0,z_1,\ldots,z_{k+1}) \in \mathbb C^{n+1} \ | \ \|z_1\|^2+\cdots+\|z_k\|^2-\|z_{k+1}\|^2<0\}\right).$$
Then it can be easily shown that the stabilizer group of $\mathbf H_{\mathbb C}^k$ in $\Sp$ is $\mathrm{Sp}(n-k)\oplus \mathrm{U}(k,1)$. Hence if $k \neq n$, then $\Gamma$ is not irreducible. Therefore $k=n$, which implies that $\Gamma$ is conjugate to a subgroup of $\mathrm{U}(n,1)$.
\end{proof}

Set $\mathbf{0}=[0,1,-1] \in \mathbb P(V_0)=\partial \mathbf H_{\H}^2$ and $\infty=[0,1,1]$. Then we recover the result of \cite{KK18}.

\begin{corollary}\label{cor:2}
Let $\Gamma$ be a nonelementary discrete subgroup of $\mathrm{Sp}(2,1)$ containing a loxodromic element fixing $\mathbf 0$ and $\infty$. If the trace field of $\Gamma$ is contained in a maximal abelian subfield of $\H$, then $\Gamma$ is conjugate to a subgroup of $\mathrm{U}(2,1)$.
\end{corollary}
\begin{proof}
By Theorem \ref{thm:main}, $\Gamma$ stabilizes a totally geodesic submanifold isometric to $\mathbf H_{\mathbb C}^1$ or $\mathbf H_{\mathbb C}^2$.
If $\Gamma$ stabilizes a totally geodesic submanifold isometric to $\mathbf H_{\mathbb C}^2$, it immediately follows that $\Gamma$ is conjugate to a subgroup of $\mathrm{U}(2,1)$.
We now suppose that $\Gamma$ preserves a totally geodesic submanifold isometric to $\mathbf H_{\mathbb C}^1$. 
As mentioned at the beginning of Section \ref{sec:proof}, we may assume that the trace of each element of $\Gamma$ is a complex number.
Noting that any loxodromic element fixing $\mathbf 0$ and $\infty$ stabilizes the unique $\mathbf H_{\mathbb C}^1$ defined as 
$$\{ [z_1,z_2,z_3] \in \mathbb P(V_-)=\mathbf H_{\H}^2 \ | \ z_1=0 \},$$
one can show that $\Gamma$ is a subgroup of $\mathrm{Sp}(1)\oplus \mathrm{U}(1,1)$.
Any element of $\mathrm{Sp}(1)\oplus \mathrm{U}(1,1)$ with complex trace is contained in $\mathrm{U}(1)\oplus \mathrm{U}(1,1)$.
Thus $\Gamma$ is a subgroup of $\mathrm{U}(1)\oplus \mathrm{U}(1,1) \subset \mathrm{U}(2,1)$.
\end{proof}

\end{document}